\documentclass[11pt,a4paper]{amsart}
\newcommand*\patchAmsMathEnvironmentForLineno[1]{%
  \expandafter\let\csname old#1\expandafter\endcsname\csname #1\endcsname
  \expandafter\let\csname oldend#1\expandafter\endcsname\csname end#1\endcsname
  \renewenvironment{#1}%
     {\linenomath\csname old#1\endcsname}%
     {\csname oldend#1\endcsname\endlinenomath}}%
\newcommand*\patchBothAmsMathEnvironmentsForLineno[1]{%
  \patchAmsMathEnvironmentForLineno{#1}%
  \patchAmsMathEnvironmentForLineno{#1*}}%
\AtBeginDocument{%
\patchBothAmsMathEnvironmentsForLineno{equation}%
\patchBothAmsMathEnvironmentsForLineno{align}%
\patchBothAmsMathEnvironmentsForLineno{flalign}%
\patchBothAmsMathEnvironmentsForLineno{alignat}%
\patchBothAmsMathEnvironmentsForLineno{gather}%
\patchBothAmsMathEnvironmentsForLineno{multline}%
}

\usepackage{amsthm, amsfonts, amssymb, amsmath, latexsym, enumerate, array,caption,subcaption}
\usepackage{amscd}     %commutative diagrams
\usepackage[all]{xy}
\usepackage[french]{babel}
\usepackage[pdftex]{graphicx}

\usepackage[latin1]{inputenc}
\usepackage{hyperref,cite,mdwlist,mathrsfs,color}

\usepackage[left=3.3cm,top=2.7cm,right=3.3cm]{geometry}

\newtheorem{thm}{Th\'eor\`eme} \newtheorem{lem}{Lemme}

 \newtheorem{prop}{Proposition}

\theoremstyle{definition}
 \newtheorem{defn}[thm]{D\'efinition}
\newtheorem{rem}{Remarque}

 \newcommand{\C}{\mathbb C}

\newcommand{\G}{\mathbb G} 
 \newcommand{\p}{\mathbb P}
\SelectTips{cm}{} \newdir{ >}{{}*!/-5pt/@{>}}
\numberwithin{equation}{section}

\allowdisplaybreaks[1]
\begin{document}
%\linenumbers
%\linenumberdisplaymath

%%%%%%%%%%%%%%%%%%%%%%%%%%%%%%%%%%%%%%%%%%%%%%%%%%%%%%%%%%%%%%%%%%%%%%%%%%%%%%
%%%%%%%%%%%%%%%%%%%% Author(s) and Address %%%%%%%%%%%%%%%%%%%%%%%%%%%%%%%%%%%
%%%%%%%%%%%%%%%%%%%%%%%%%%%%%%%%%%%%%%%%%%%%%%%%%%%%%%%%%%%%%%%%%%%%%%%%%%%%%%

% ====================================================================
% Title, authors, abstract, keywords and AMS codes
% ====================================================================

\title{\`A propos des vari\'et\'es de Poncelet}

% If your title is too long to fit on the running head:
% \title[short title]{Title of your article}

% If your title is too long to fit on one line:
% \title[short title]{Title first line  \\ title second line}

\author{Jean Vall\`es}
\email{{\tt jean.valles@univ-pau.fr}}
\address{Universit\'e de Pau et des Pays de l'Adour \\
  Avenue de l'Universit\'e - BP 576 - 64012 PAU Cedex - France}


% If there are two authors:
% \author{First author and second author}
\thanks{L'auteur a \'et\'e  partiellement financ\'e par les  ANR-09-JCJC-0097-0 INTERLOW et ANR GEOLMI.}
% If there are three or more authors:
% \author{First author, second author,...,nth author and last author}

% If the authors' names do not fit on one line:
% \author[abbreviated list]{first line \\ second line}
% In the optional argument put something like first author et al.

\begin{abstract}
On d\'efinit les vari\'et\'es de Poncelet. Celles-ci g\'en\'eralisent tr\`es naturellement les c\'el\`ebres courbes de Poncelet, qui, conform\'ement \`a l'axiome d'Arnold, 
ne furent pas introduites par Poncelet mais par Darboux. 
On montre ensuite qu'une surface g\'en\'erale de degr\'e $\ge 4$ dans $\p^3$ n'est pas une surface de Poncelet mais,
par contre, que  toutes les  quadriques et toutes les  cubiques lisses de $\p^3$ sont des  surfaces de Poncelet. 
\end{abstract}
\maketitle

% ====================================================================
% Section 1
% ====================================================================

\section{Introduction}\label{sec1}

La petite histoire des math\'ematiques croise souvent l'histoire plus rocambolesque des livres scolaires.  Marie Stuart  fut d\'ecapit\'ee pour son go\^ut du chiffre, go\^ut qu'elle acquit \`a la cour
 du roi Charles IX aupr\`es de Vi\`ete,
Pascal d\'ecrit l'\textit{esprit g\'eom\'etrique} pour instruire les petits \'el\`eves jans\'enistes  avant que Louis XIV ne condamne Port-Royal,  et,  
plus pr\`es de nous, Poncelet, jeune lieutenant de la grande arm\'ee de Napol\'eon posa les bases de la g\'eom\'etrie projective moderne dans une prison russe
o\`u il fut retenu prisonnier apr\`es la bataille de Krasnoie, bataille rest\'ee c\'el\`ebre car l'empereur des fran\c{c}ais r\'eussit \`a passer entre les mailles serr\'ees du filet 
tendu par Koutouzov. Dans les 
\textit{M\'emoires d'outre-tombe}  Chateaubriand (cf. \cite{Ch}, page 211, tome 2) relate l'\'episode avec emphase :
``\textit{ Les hauteurs environnantes au pied desquelles marchait Napol\'eon, se chargeaient d'artillerie et pouvaient \`a chaque instant le foudroyer ; il y jette 
un coup d'\oe{}il et dit :} ``Qu'un escadron de mes chasseurs s'en empare!''\textit{ Les Russes n'avaient qu'\`a se laisser rouler en bas, leur masse l'e\^ut \'ecras\'e ; 
mais, \`a la vue de ce grand homme et des d\'ebris de la garde serr\'ee en bataillon carr\'e, ils demeur\`erent immobiles, comme fascin\'es : son regard arr\^eta cent mille hommes sur les collines.}'' 
 
Poncelet, quant \`a lui, resta sur le champ de bataille, ``[...] \textit{d'o\`u il n'est sorti vivant que par une faveur sp\'eciale de Dieu, au milieu de ses chefs, de ses camarades tu\'es ou atteints de blessures, 
toutes mortelles 
dans ce pernicieux climat}'' (cf. \cite{P}, Pr\'eface), et fut fait prisonnier le lendemain de la fuite de Napol\'eon.

Dans le manuscrit r\'edig\'e lors de son s\'ejour russe \cite{P}, il d\'emontre qu'une conique $n$-circonscrite \`a une conique lisse fix\'ee lui est infiniment circonscrite.
C'est le  th\'eor\`eme appel\'e th\'eor\`eme de cl\^oture. Quelques d\'ecennies  plus tard, Darboux a montr\'e qu'une courbe de degr\'e $n-1$ passant par les $\binom{n}{2}$ sommets  
d'un polygone dont les $n$ c\^ot\'es sont tangents \`a une conique lisse $C$ passe par une infinit\'e de sommets de tels polygones (cf. \cite{Da}, p. 248).
 C'est aussi dire que tout point de la courbe est le sommet d'une configuration 
de $n$ droites tangentes \`a $C$. Ces courbes  sont appel\'ees \textit{courbes de Poncelet}.
%\begin{figure}[tp]
%  \centering
%  \framebox[0.6\textwidth]{\rule[-9ex]{0pt}{19ex}Your graphic would be here}
%   \includegraphics[some options]{ponceletcubic.pdf}
%  \caption{Example of a  figure.}\label{fig1}
%\end{figure}
 \begin{figure}[h!]
    \centering
    \includegraphics[height=6.5cm]{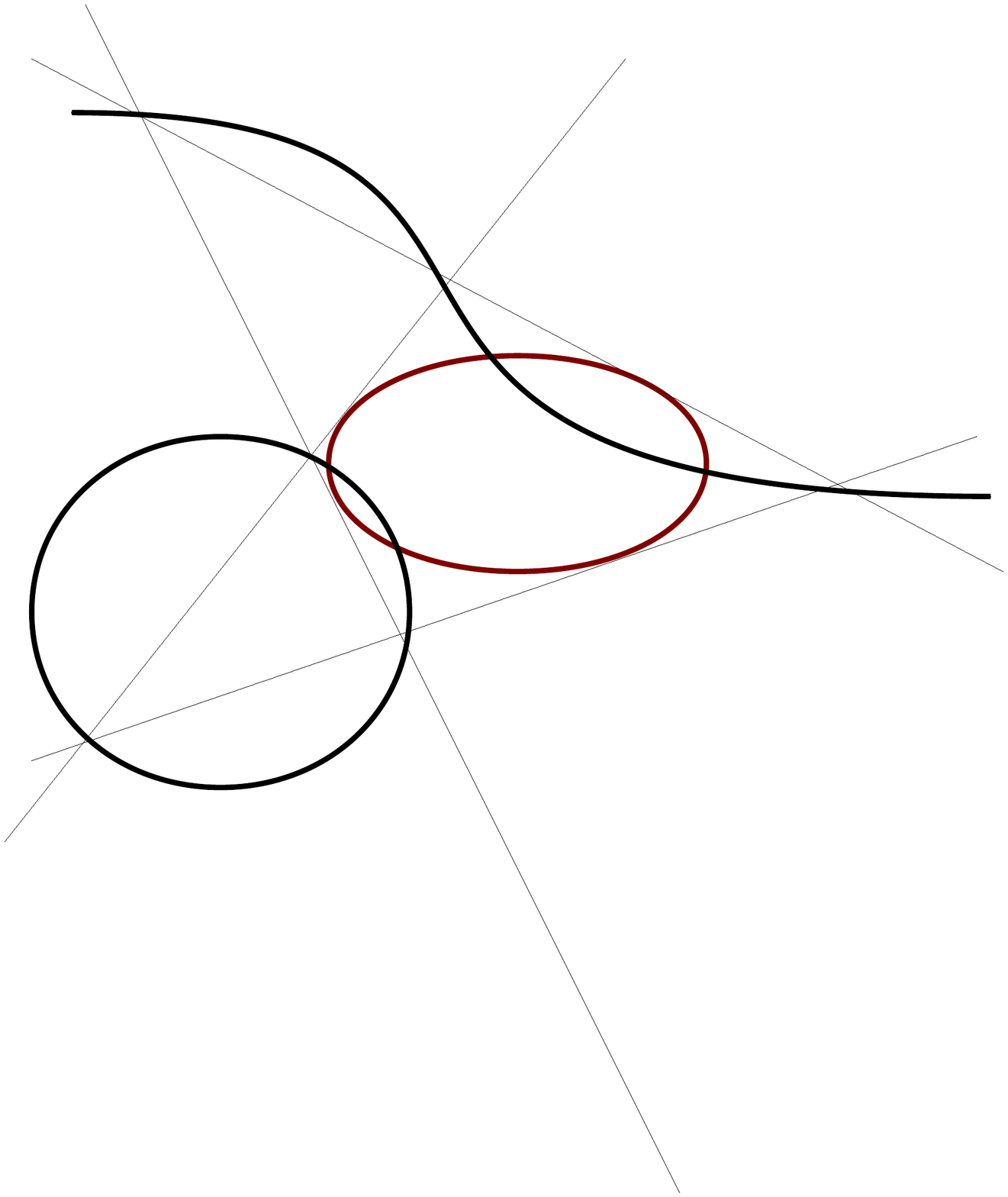}
    % \qquad \qquad
    % \includegraphics{tube}
    \caption{Un quadrilat\`ere complet circonscrit et une cubique de  Poncelet}
  \end{figure}
Pour les vari\'et\'es de dimensions plus grandes introduites par Trautmann dans \cite{T}, et que l'on red\'efinit ci-dessous comme d\'eterminant de sections d'un certain fibr\'e vectoriel 
(voir la d\'efinition \ref{sp}), 
 le th\'eor\`eme de Darboux ne s'applique plus; plus pr\'ecis\'ement,
 il ne suffit plus de passer par les sommets d'un polytope $(n-1)$  tangents \`a la courbe rationnelle normale sous-jacente pour que chacun de ses points soit le sommet d'un tel polytope. 
On v\'erifie ci-dessous  que pour qu'une hypersurface soit une hypersurface de Poncelet, il lui faudra,  passer 
par les sommets d'une infinit\'e de dimension $n-1$ de tels polytopes (voir la proposition \ref{hypPoncelet}).

L'espace projectif des courbes planes de degr\'e $n-1$ passant par les $\binom{n}{2}$ sommets de $n$ droites fix\'ees est de dimension $n-1$. La donn\'ee de $n$ droites tangentes \`a une conique lisse 
\'equivaut \`a la donn\'ee d'un diviseur de degr\'e $n$ sur $\p^1$. Sa dimension projective est $n$. Enfin l'espace des coniques du plan est de dimension projective $5$. Ainsi 
un d\'ecompte na\"{\i}f de dimension  du  sch\'ema des courbes planes de degr\'e $n-1$ qui sont de Poncelet donne $(n-1)+n+5=2n+4$.
Mais ce d\'ecompte est faux d'apr\`es  le th\'eor\`eme de Darboux. On v\'erifie en effet que cette dimension est $2n+3$, ce qui correspond au choix d'une conique (dimension $5$)
et d'un pinceau de formes binaires de degr\'e $n$ (dimension de $\G(1,n)$ \'egale $2n-2$) sur cette conique. Ainsi, contrairement aux cubiques planes et aux coniques (cf. \cite{Ba}, p. 90 et p. 85),
 la vari\'et\'e des quartiques de Poncelet
(appel\'ees  quartiques de L\"uroth) ne remplit pas l'espace des quartiques planes et forme un diviseur. 
Morley a prouv\'e que le degr\'e de la sous-vari\'et\'e des quartiques de L\"uroth \'etait \'egal \`a $54$ (cf. \cite{Mo}).
 Barth, dans son article 
fondateur sur les fibr\'es vectoriels de rang deux sur le plan (cit\'e ci-dessus), montre que cette hypersurface est l'image de l'espace de modules $\mathrm{M}_{\p^2}(0,4)$ (des faisceaux localement libres et 
semistables de premi\`ere et seconde classes de Chern \'egales \`a $0$ et $4$) par l'application qui \`a un fibr\'e associe sa courbe de droites de saut. Qui plus est, le degr\'e de l'image
est le nombre de Donaldson\footnote{Ces nombres, et plus g\'en\'eralement les invariants de Donaldson, ont \'et\'e introduits par Donaldson lui-m\^eme, en $1983$, un an 
apr\`es que la classification des vari\'et\'es topologiques compactes simplement connexes de dimension $4$ a \'et\'e \'etablie  par Freedman \cite{Fr}. Donaldson, en introduisant ces invariants, a montr\'e
 qu'il existe  des surfaces complexes lisses
hom\'eomorphes qui ne sont pas diff\'eomorphes \cite{Do}.}  $q_{13}$,  ce qui accro\^{\i}t encore   l'int\'er\^et
pour le diviseur des quartiques de L\"uroth. En effet, comme le morphisme de Barth est g\'en\'eriquement injectif,  ce degr\'e  co\"{\i}ncide avec $q_{13}$ (cf. \cite{LP} et \cite{OS}, ainsi bien-s\^ur que \cite{Mo}).

Le m\^eme probl\`eme  se pose en dimension plus grande. \`A un couple $(C_n, f_1\wedge \cdots \wedge f_n))$ form\'e d'une courbe rationnelle normale et de $n$ formes binaires de degr\'e $n+k$ 
on associe une hypersurface de Poncelet de degr\'e $k+1$ sur $\p^n$. Un d\'ecompte des dimensions attendues montre que cette application ne peut pas \^etre dominante pour $n=3$ d\`es que $k\ge 3$.
Par contre elle peut l'\^etre lorsque $n=3$ pour $k=1$ et $k=2$. Et en effet, on montre qu'elle l'est, plus pr\'ecis\'ement, on montre qu'une surface quadrique quelconque ou  une surface cubique  lisse  de $\p^3$ 
sont des  surfaces de Poncelet 
(voir les th\'eor\`emes \ref{surfacequadrique} et \ref{surfacecubique}).
% ====================================================================
% Section 2
% ====================================================================
\section{G\'en\'eralit\'es et notations}\label{sec2}
Soit $V$ un espace vectoriel complexe de dimension $2$. On note $S_i$ les espaces vectoriels $\mathrm{Sym}^i(V)$ (et $S_i^{\vee}=\mathrm{Sym}^i(V^{\vee})$) 
invariants sous l'action de $\mathrm{SL(V)}$. La multiplication des formes binaires :
$$ \begin{CD}
S_n^{\vee}\otimes S_k^{\vee} @>>> S_{n+k}^{\vee}
   \end{CD}$$
permet de d\'efinir deux fibr\'es, dits  de Schwarzenberger, $E_{k,n+k}$ sur $\p( S_k^{\vee} ) $ et $E_{n,n+k}$ sur $\p(S_n^{\vee} )$
ainsi que la vari\'et\'e des espaces lin\'eaires  $N$ s\'ecants (avec $N=\mathrm{inf}(k,n)+1$) \`a la courbe $C_{n+k}$, image de 
$\p(V)$ dans $\p( S_{n+k})$ par le morphisme de Veronese. Les fibr\'es projectivis\'es sont, pour l'un, l'\'eclatement de  $\p( S_{n+k})$ le long de cette 
vari\'et\'e et  pour l'autre, le mod\`ele lisse de cette vari\'et\'e de s\'ecantes, obtenu en l'\'eclatant le long de son lieu singulier.
Tout ceci est connu et r\'edig\'e en plusieurs endroits (cf. par exemple \cite{ISV}).\\

Dans l'article \cite{V} consacr\'e aux fibr\'es de Schwarzenberger,  ces courbes sont interpr\'et\'ees comme des pinceaux de sections de ces fibr\'es.
C'est cette interpr\'etation que nous reprenons ci-dessous en la g\'en\'eralisant aux dimensions plus grandes.\\

Consid\'erons le fibr\'e de Schwarzenberger suivant :
$$\begin{CD}
0 @>>> S_k\otimes \mathcal{O}_{\p(S_n^{\vee})}(-1)@>>>  S_{n+k}
\otimes  \mathcal{O}_{\p(S_n^{\vee})} @>>>
E_{n,n+k}@>>> 0.
\end{CD}$$
Nous associons, \`a un syst\`eme lin\'eaire  $\Lambda=<f_0,\cdots, f_r> \subset S_{n+k}$ de dimension $r+1\le n$,  un sous-sch\'ema de $\p(S_n^{\vee})$ d\'efini  
par les \'equations $f_0\wedge\dots\wedge f_r=0$ (l'\'ecriture est abusive mais nous l'expliciterons plus loin).
En effet, 
comme
$
\mathrm{H}^0(E_{n, n+k})\cong \mathrm{H}^0(\mathcal{O}_{\p^1}(n+k)),$
les formes binaires  $f_i$ s'interpr\`etent comme des  sections de  $E_{n,n+k}$.  Le diagramme commutatif suivant l'explique :
$$
 \begin{CD}
@.@. \Lambda\otimes\mathcal{O}_{\p(S_n^{\vee})} @ =  \Lambda\otimes\mathcal{O}_{\p(S_n^{\vee})}\\
@. @. @VVV   @ VVV\\
 0@>>> S_k\otimes \mathcal{O}_{\p(S_n^{\vee})}(-1)@>>>S_{n+k}\otimes \mathcal{O}_{\p(S_n^{\vee})}  @>>> E_{n,n+k}@>>>0\\
  @.                           @|              @ VVV         @VVV                         \\
0 @ >>> S_k\otimes \mathcal{O}_{\p(S_n^{\vee})}(-1)@>>>\frac{S_{n+k}}{\Lambda}\otimes \mathcal{O}_{\p(S_n^{\vee})}   @>>> {\mathcal L}@>>>0.
\end{CD}
$$
Le  faisceau  ${\mathcal L}$, d\'efini sur $\p(S_n^{\vee})$, cesse d'\^etre localement libre le long de la sous-vari\'et\'e ferm\'ee d'\'equation $f_0\wedge\dots\wedge f_r=0$.

\begin{defn}
\label{sp}
Les vari\'et\'es  d\'efinies par la donn\'ee de $r\le n$ sections lin\'eairement ind\'ependantes d'un fibr\'e de  Schwarzenberger  $E_{n,n+k}$
au dessus de  $\p(S_n^{\vee})$ sont appel\'ees  {\em Vari\'et\'es de Poncelet}.
\end{defn}
Une section de $E_{n,n+k}$ correspond \`a la donn\'ee de  $n+k$ points sur la courbe rationnelle normale $C_{n+k}$ et, par l'isomorphisme canonique $C_{n+k}\simeq C_n\simeq \p(S_1)$, sur la courbe rationnelle normale 
$C_{n}\subset \p(S_n)$; cette section s'annule le long des 
 $\binom{n+k}{n}$
points d'intersection de   $n$ hyperplans 
qui ont un contact d'ordre $n-1$ avec   $C_{n}$ le long de  $n$ points choisis parmi les  $n+k$ points.
Ces vari\'et\'es de Poncelet sont ainsi caract\'eris\'ees par le fait de passer par les sommets  
d'un polytope \`a  $n+k$ faces ; celles-ci sont des hyperplans tangents d'ordre $n-1$ \`a $C_n$.
Le cas des courbes est bien connu (cf. par exemple \cite{Da} et \cite{T}). En dimension sup\'erieure, hormis quelques lignes 
\`a la  fin de l'article \cite{T} et quelques exemples dans \cite{ISV}, elles n'ont pas \'et\'e \'etudi\'ees.

La proposition suivante g\'en\'eralise le th\'eor\`eme de Darboux concernant les courbes. 

\begin{prop}
\label{hypPoncelet}
 Une hypersurface de degr\'e $k+1$ est une hypersurface de Poncelet sur $\p^n$ si et seulement si elle contient une vari\'et\'e de Poncelet de codimension 
$2$  d\'efinie par 
$(n-1)$ sections lin\'eairement ind\'ependantes du fibr\'e $E_{n,n+k}$.
\end{prop}
\begin{proof}
Consid\'erons le diagramme commutatif suivant :
$$
 \begin{CD}
@.@. \mathcal{O}_{\p(S_n^{\vee})}^{n-1} @ =  \mathcal{O}_{\p(S_n^{\vee})}^{n-1}\\
@. @. @VVV   @ VVV\\
 0@>>> S_k\otimes \mathcal{O}_{\p(S_n^{\vee})}(-1)@>>>S_{n+k}\otimes \mathcal{O}_{\p(S_n^{\vee})}  @>>> E_{n,n+k}@>>>0\\
  @.                           @|              @ VVV         @VVV                         \\
0 @ >>> S_k\otimes \mathcal{O}_{\p(S_n^{\vee})}(-1)@>>>\frac{S_{n+k}}{\C^{n-1}} \otimes \mathcal{O}_{\p(S_n^{\vee})}   @>>> {\mathcal I}_{X}(k+1)@>>>0.
\end{CD}
$$
Le sch\'ema $X$ est de dimension $n-2$ et il est  d\'efini par $n-1$ sections ind\'ependantes de $E_{n,n+k}$.
Une section non nulle de $\mathrm{H}^0({\mathcal I}_{X}(k+1))$, c'est-\`a-dire une hypersurface de degr\'e $k+1$ passant par les sommets des 
$n-1$ polytopes d\'etermin\'es par les $n-1$ sections, provient d'une section suppl\'ementaire de $E_{n,n+k}$. En effet cela r\'esulte de la surjectivit\'e de  
l'application $$ \mathrm{H}^0(E_{n,n+k}) \rightarrow \mathrm{H}^0({\mathcal I}_{X}(k+1)).$$ 
 L'hypersurface est alors d\'efinie par $n$ sections, autrement dit, son \'equation 
est le  d\'eterminant de ces $n$ sections.
\end{proof}
\section{Surfaces de Poncelet}
Les cubiques planes sont toutes de Poncelet tandis que les quartiques planes de Poncelet ne forment qu'un diviseur. Plus g\'en\'eralement, les courbes de Poncelet de degr\'e $k$ forment une sous-vari\'et\'e de 
dimension $2(k+1)+3$ de l'espace des courbes de degr\'e $k$. Pour les hypersurfaces de Poncelet de $\p^n$  de  degr\'e $k$  la dimension attendue est 
$$\mathrm{dim}\G(k,n+k-1)+\mathrm{dim}\frac{\mathrm{PGL}(n+1)}{\mathrm{PGL}(2)}.$$ 
Ainsi   la dimension attendue des surfaces de Poncelet de degr\'e $k$ est $3k+12$. Ce qui fait $18$  pour les quadriques, $21$ pour les cubiques, $24$ pour les quartiques.
L'espace de toutes les quadriques et de dimension $9$, celui des cubiques de dimension $19$ et celui des quartiques $34$. Pour $k\ge 4$, pour des raisons \'evidentes de dimension, une surface de degr\'e $k$ 
g\'en\'erale n'est pas de Poncelet. Par contre on attend a priori  que toute surface quadrique et toute surface 
cubique soient de Poncelet (et plut\^ot deux fois qu'une). C'est ce que nous v\'erifions maintenant.

Le cas des quadriques est \'el\'ementaire.
\begin{thm}
\label{surfacequadrique}
 Toute quadrique de $\p^3$ est une quadrique de Poncelet.
\end{thm}
\begin{proof}
 Une quadrique de Poncelet est une section du fibr\'e de Schwarzenberger d\'efini par la suite exacte :
$$ \begin{CD}
 0@>>> S_1\otimes \mathcal{O}_{\p(S_3^{\vee})}(-1)@>M>>S_4\otimes \mathcal{O}_{\p(S_3^{\vee})}  @>>> E_{3,4}@>>>0.
   \end{CD}
$$
Quitte \`a noter $\p(S_3)=\mathrm{proj}(\C[x_0,x_1,x_2,x_3])$ la matrice $M$ est bien d\'ecrite :
$$ M= \left (
             \begin{array}{cc}
              x_0 &  0       \\
              x_1 & x_0    \\
              x_2 & x_1     \\
              x_3 & x_2  \\
                0  & x_3  
             \end{array}
      \right ).
$$
Les mineurs $ \left |
             \begin{array}{cc}
              x_1 & x_0    \\     
              x_3 & x_2            
             \end{array}
      \right |
$, $ \left |
             \begin{array}{cc}
              x_1 & x_0    \\     
              x_2 & x_1            
             \end{array}
      \right |
$, $ \left |
             \begin{array}{cc}
              x_1 & x_0    \\     
              0 & x_3            
             \end{array}
      \right |
$ et $ \left |
             \begin{array}{cc}
              x_0 & 0   \\     
              x_1 & x_0            
             \end{array}
      \right |
$
extraits de $M$ d\'efinissent des  quadriques de rangs respectivement \'egaux \`a $4$, $3$, $2$ et $1$. 
\end{proof}
Le cas des cubiques est un peu plus difficile.
\begin{thm}
\label{surfacecubique}
Toute surface cubique lisse de $\p^3$ est une surface de Poncelet.
\end{thm}
\begin{rem}
Il est certainement  int\'eressant de d\'eterminer les cubiques singuli\`eres qui ne sont pas de Poncelet. Par exemple, je pense que trois plans 
poss\'edant une droite commune ne sont pas une surface de Poncelet. 
\end{rem}
\begin{proof}
Soit  $S$ une surface cubique lisse. Elle contient $27$ droites exactement et on peut en choisir $6$ droites parmi elles qui soient deux \`a deux disjointes.
En  contractant ces $6$ droites on obtient un $\p^2$ et chacune des $6$ droites a pour image un point ; de plus les $6$ points ainsi obtenus, que l'on note $Z$,  sont en position g\'en\'erale.

\smallskip

C'est classique. Les six points $Z$ (et m\^eme le double six) et la surface cubique $S$ proviennent d'un m\^eme $3$-tenseur
$\phi \in A\otimes B\otimes C$ avec $\mathrm{dim}_{\C}(A)=\mathrm{dim}_{\C}(B)=3$ et $\mathrm{dim}_{\C}(C)=4$. Sur $\p(C)=\p^3$, ce tenseur devient  une matrice  de formes lin\'eaires
$$ \begin{CD}
    A^{\vee}\otimes \mathcal{O}_{\p(C)}(-1) @>\phi>> B\otimes \mathcal{O}_{\p(C)}, 
   \end{CD}
  $$
dont le d\'eterminant est une \'equation de $S$. Sur $\p(B)=\p^2$, ce m\^eme tenseur devient une matrice de 
formes lin\'eaires
$$ \begin{CD}
    A^{\vee}\otimes \mathcal{O}_{\p(B)}(-1) @>\phi>> C\otimes \mathcal{O}_{\p(B)}, 
   \end{CD}
  $$
dont les mineurs maximaux engendrent   $\mathrm{H}^0(\mathcal{I}_Z(3))$, o\`u  $\mathcal{I}_Z$ est l'id\'eal des six points. L'\'eclatement du plan le long de $Z$,
 \`a savoir $\p(\mathcal{I}_Z(3))\subset \p(B)\times \p(C)$, s'envoie isomorphiquement sur $S$ par la seconde projection.

\smallskip

Pour montrer  que $S$ est une surface  de Poncelet, il suffit de montrer que $S$ est d\'efinie par un tenseur $\phi  \in S_2\otimes S_3\otimes  (\frac{S_{5}}{A})^{\vee}$ o\`u $A=\C^3$ est 
un sous espace vectoriel de $S_5$. Remarquons tout d'abord que si l'on consid\`ere le groupe de $6$ points $Z$ associ\'es au tenseur : 
$$ \begin{CD}                                              \\
0 @ >>> S_2\otimes \mathcal{O}_{\p(\frac{S_{5}}{A})}(-1)@>>> S_3\otimes \mathcal{O}_{\p(\frac{S_{5}}{A})}   @>>> {\mathcal I}_{Z}(3)@>>>0,
   \end{CD}
$$
 la matrice des relations est persym\'etrique (ou de Haenkel). En effet, c'est la restriction \`a $\p(\frac{S_{5}}{A})$ de la matrice persym\'etrique canonique (cf. par exemple \cite{Ha}) de dimension
$3\times 4$ sur $\p(S_5)$.

\smallskip

 R\'eciproquement une  matrice $3\times 4$ persym\'etrique de formes lin\'eaires sur $\p^2$ est toujours la restriction de la matrice persym\'etrique canonique sur $\p^5$ (puisque d\`es que les 
nombres de lignes et de colonnes sont fix\'es, il n'existe qu'une seule matrice persym\'etrique g\'en\'erique). Or les mineurs $2\times 2$ de cette matrice persym\'etrique g\'en\'erique
d\'efinissent une  quintique rationnelle normale $C_5\subset \p^5$, et via 
 l'association d\'ecrite ci-dessus,  une surface cubique de Poncelet. Le th\'eor\`eme sera donc prouv\'e lorsque  le lemme suivant le sera.
\begin{lem}
\label{gruson}
 Six points en position lin\'eaire g\'en\'erale sont r\'esolus par une matrice persym\'etrique.
\end{lem}
\begin{proof}[D\'emonstration du lemme]
Consid\'erons six points en position g\'en\'erale sur le plan. Le syst\`eme lin\'eaire des courbes de degr\'e $5$ passant doublement par chacun des points
 a dimension projective $2$ ($(21-6\times 3)-1$). Une courbe g\'en\'erale $\Gamma_5$ de ce syst\`eme est lisse hors des points bases et elle a genre $0$. Elle est donc la projection
d'une quintique rationnelle lisse de $C_5\subset \p^5$ sur la quintique $\Gamma_5\subset \p^2=\p(A^{\vee})$.   En notant $S_i$
l'ensemble des formes binaires de degr\'e $i$ sur $C_5\simeq \p^1$, on a $C_5\subset \p(S_5)$ et il existe un sous espace vectoriel 
$A\subset S_5$ de dimension $3$ qui d\'efinit le centre de projection. 
Plus pr\'ecis\'ement,  la suite exacte suivante d'espaces vectoriels :
$$ \begin{CD}                                              \\
0 @ >>> A @>>> S_5 @>>> \frac{S_{5}}{A} @>>>0,
   \end{CD}
$$
induit la morphisme de projection : $\pi  : \p(S_5)\setminus \p(\frac{S_{5}}{A})\rightarrow \p(A^{\vee})$.
 La projection restreinte $\pi_{|C_5} : C_5 \longrightarrow \Gamma_5$ est ramifi\'ee au-dessus des six points doubles de $\Gamma_5$.
(Les $6$ points doubles de la projection sont les images des $6$  points d'intersection de $\p(\frac{S_{5}}{A})$ avec la vari\'et\'e des droites bis\'ecantes \`a $C_5$.) Les points doubles de la
quintique $\Gamma_5$ sont d\'efinis par le lieu de d\'eg\'en\'erescence de la matrice
de multiplication (par un \'el\'ement variable de $S_2$) de $S_3$ dans $\frac{S_{5}}{A}$, ou
encore un tenseur dans $S_2\otimes  S_3 \otimes (\frac{S_{5}}{A})^{\vee}$. Il est persym\'etrique.
\end{proof}
\textit{Reprise de la d\'emonstration du th\'eor\`eme.}
Le tenseur  $S_2\otimes  S_3 \otimes (\frac{S_{5}}{A})^{\vee}$ d\'efinit aussi la surface cubique $S\subset \p(S_3)$ (plus exactement un faisceau ${\mathcal L}_{S}$ de rang $1$ 
 sur $S$) associ\'ee aux six points, ainsi que le fibr\'e de Schwarzenberger dont elle est le d\'eterminant de trois sections ind\'ependantes,
comme le montre le diagramme commutatif  suivant :
$$ \begin{CD}
@.@. A\otimes\mathcal{O}_{\p(S_3^{\vee})} @ =  A\otimes\mathcal{O}_{\p(S_3^{\vee})}\\
@. @. @VVV   @ VVV\\
 0@>>> S_2\otimes \mathcal{O}_{\p(S_3^{\vee})}(-1)@>>>S_5\otimes \mathcal{O}_{\p(S_3^{\vee})}  @>>> E_{3,5}@>>>0\\
  @.                           @V=VV              @ VVV         @VVV                         \\
0 @ >>> S_2\otimes \mathcal{O}_{\p(S_3^{\vee})}(-1)@>>>\frac{S_{5}}{A} \otimes \mathcal{O}_{\p(S_3^{\vee})}   @>>> {\mathcal L}_{S}@>>>0.
   \end{CD}
$$
\end{proof}
\begin{rem} 
Pour \'etablir ce dernier th\'eor\`eme, on peut aussi  montrer  que la surface cubique $S$ contient une courbe de Poncelet de genre $3$ et de degr\'e $6$.
Cette courbe est le plongement dans $\p^3$ d'une quartique plane lisse ; est-elle une quartique sp\'eciale, une quartique de Poncelet par exemple ?
\end{rem}

% ====================================================================
% Acknowledgements
% ====================================================================

 Je remercie le rapporteur anonyme pour cette derni\`ere remarque et Laurent Gruson qui m'a donn\'e la preuve du lemme \ref{gruson}. 

% ====================================================================
% References
% ====================================================================

% 1. Build your bib file and write its name below instead of bibmodel.
% 2. Run pdfLaTeX on the tex file.
% 3. Run BibTeX on the aux file.
% 4. Run LaTeX again a couple of times to get the correct 
%    cross-references.
%
% In many LaTeX editors, the above loop can be done just by pressing 
% a couple of buttons. 
  
\bibliographystyle{acmurl}
\bibliography{bibliography}
   
% ====================================================================
% Address
% ====================================================================

% Single author 

% Two authors with the same address

% \begin{address}
%   first author and second author \\
%   first line of the address \\
%   second line \\
%   last line \\
%   \texttt{first author's email} and \texttt{second author's email}
% \end{address}

% Two authors with different addresses: first option

% \begin{address}
%   first author \\
%   first line of the address \\
%   second line \\
%   last line \\
%   \texttt{first author's email}
% \end{address}
% 
% \begin{address}
%   second author \\
%   first line of the address \\
%   second line \\
%   last line \\
%   \texttt{second author's email}
% \end{address}

% Two authors with different addresses: second option

% \begin{address}[2]
%   first author \\
%   first line of the address \\
%   second line \\
%   last line \\
%   \texttt{first author's email}
% \end{address}
% \begin{address}
%   second author \\
%   first line of the address \\
%   second line \\
%   last line \\
%   \texttt{second author's email}
% \end{address}

% Three authors. There are many cases. The simplest way to proceed: 
% put every author in a separate address environment. You can also,
% for example:
% 
% * join two consecutive addresses on the same line (see the second
%   option for two authors with different addresses)
%   
% * join several authors with the same address (see two authors
%   with the same address)

\end{document}